\documentclass[12pt]{amsart}

\usepackage{amssymb}
\usepackage{amscd}
\usepackage{epsfig}
\usepackage{color}

\usepackage[setpagesize=false]{hyperref}

\usepackage{ascmac}
\usepackage[all]{xy}
\usepackage{time}

\usepackage[OT2,T1]{fontenc}
\DeclareSymbolFont{cyrletters}{OT2}{wncyr}{m}{n}
\DeclareMathSymbol{\Sha}{\mathalpha}{cyrletters}{"58}

\usepackage{stackengine}
\stackMath
\newcommand\undertilde[2][1]{%
 \def\useanchorwidth{T}%
  \ifnum#1>1%
    \stackunder[0pt]{\tenq[\numexpr#1-1\relax]{#2}}{\scriptscriptstyle\sim}%
  \else%
    \stackunder[1pt]{#2}{\scriptscriptstyle\sim}%
  \fi%
}

\DeclareMathAlphabet{\mathpzc}{OT1}{pzc}{m}{it}

\usepackage{yhmath}
\usepackage{shuffle}

\usepackage{mathrsfs}

\usepackage{ulem}

\title[Kashiwara-Vergne and double shuffle Lie algebras]{Notes on Kashiwara-Vergne 
and double shuffle Lie algebras} 

\author{Hidekazu Furusho}
\author{Nao Komiyama}

\address{Graduate School of Mathematics, Nagoya University, 
Furo-cho, Chikusa-ku, Nagoya, 464-8602, Japan}
\email{furusho@math.nagoya-u.ac.jp}

\address{Department of Mathematics, Graduate School of Science, Osaka University Toyonaka, Osaka 560-0043, Japan}
\email{komiyama.nao.aww@osaka-u.ac.jp}

\date{October 30, 2024}
\dedicatory{In memory of Toshie Takata.}

\newtheorem{thm}{Theorem}[section]
\newtheorem{lem}[thm]{Lemma}
\newtheorem{cor}[thm]{Corollary}
\newtheorem{prop}[thm]{Proposition}

{\theoremstyle{definition} \newtheorem{rem}[thm]{Remark}}
{\theoremstyle{definition} \newtheorem{defn}[thm]{Definition}}

\theoremstyle{definition}

{\theoremstyle{remark} }

\numberwithin{equation}{section}

\newcommand{\Q}{\mathbb{Q}}
\newcommand{\Z}{\mathbb{Z}}
\newcommand{\N}{\mathbb{N}}

\newcommand{\Sh}[3]{{\rm Sh}\binom{#1;#2}{#3}}
\newcommand{\hshuffle}{\shuffle_*}
\newcommand{\hSh}[3]{{\rm Sh}_*\binom{#1;#2}{#3}}

\newcommand{\vecx}{{\bf x}}

\newcommand{\krv}{\mathfrak{krv}}
\newcommand{\lkrv}{\mathfrak{lkrv}}

\newcommand{\dmr}{\mathfrak{dmr}}
\newcommand{\grt}{\mathfrak{grt}}

\newcommand{\der}{\mathfrak{der}}
\newcommand{\tder}{\mathfrak{tder}}
\newcommand{\sder}{\mathfrak{sder}}

\newcommand{\tr}{\mathrm{tr}}
\newcommand{\anti}{\mathrm{anti}}
\newcommand{\Fil}{\mathrm{Fil}}

\newcommand{\ma}{\mathrm{ma}}
\newcommand{\vimo}{\mathrm{vimo}}
\newcommand{\teru}{\mathrm{teru}}
\newcommand{\push}{\mathrm{push}}
\newcommand{\pus}{\mathrm{pus}}
\newcommand{\pusnu}{\mathrm{pusnu}}

\newcommand{\mantar}{\mathrm{mantar}}
\newcommand{\swap}{\mathrm{swap}}
\newcommand{\pspush}{\mathrm{sena}}

\newcommand{\ARI}{\mathrm{ARI}}

\newcommand{\ari}{\mathrm{ari}}

\newcommand{\al}{\mathrm{al}}
\newcommand{\il}{\mathrm{il}}

\newcommand{\fin}{\mathrm{fin}}

\newcommand{\id}{\mathrm{id}}
\newcommand{\D}{\mathcal{D}}

\newcommand{\h}{\undertilde{t}}
\newcommand{\tswap}{\undertilde{u}}
\newcommand{\coll}{\mathrm{coll}}

\begin{document}
\bibliographystyle{amsalpha+}
\maketitle

\begin{abstract}
We explain the current situation of the relationship between
the Kashiwara-Vergne Lie algebra $\krv$ and
the double shuffle Lie algebra $\dmr$.
We  also show the validity  of Ecalle's senary relation for small depths.
\end{abstract}

\tableofcontents

\section{Introduction}\label{introduction}

The {\it double shuffle Lie algebra} $\dmr$ (cf. Definition \ref{defn:dmr})
is the graded Lie algebra
constructed by Racinet \cite{R}.
The symbol $\dmr$ stands for \lq double m\'elange et r\'egularisation' in French,
the regularized double shuffle relations among multiple zeta values,
which serve as defining equations of the Lie algebra.
  
The {\it Kashiwara-Vergne Lie algebra} $\krv$ (cf. Definition \ref{defn:krv})
is the graded Lie algebra
introduced by Alekseev, Torossian and Enriquez \cite{AT, AET}.
It is constructed from the set of solutions of the  Kashiwara-Vergne problem.

It is expected that 
those two Lie algebras are isomorphic
(cf. \cite{F14}).
So far, we know 
that there exists  an inclusion from $\dmr$ 
to $\krv$
under a certain assumption on Ecalle's mould theory \cite{E-flex}
by Schneps's results (\cite{S} Theorem 1.1):

\begin{thm}[\cite{S}]
\label{Schneps Theorem 1.1}
Let $\tilde f\in \dmr$.
Assume that 
the following {\it senary relation}  (cf. \cite[(3.64)]{E-flex})
\begin{equation}\label{senary relation}
\teru( M)^r=\push\circ\mantar\circ\teru\circ\mantar( M)^r
\end{equation}
holds for any $r\geqslant 1$
and the mould $M=\ma_{\tilde f}$ 
under 
the map 
\begin{equation}\label{eq:tilde f to F}
\nu:\mathbb L\to \der; \quad
\tilde f\mapsto d_F.
\end{equation}
Then the image $\nu(\tilde f)$ belongs to $\krv$.
\end{thm}

As for moulds, $\teru$, $\push$, $\mantar$ and $\ma$, see \S\ref{sec:mould theory}.
For $\mathbb L$, $\der$,  $\tilde f$ and $d_F$, 
see \S \ref{sec:2LAs}
(particularly for the last two, see
\eqref{eq:F f tilde f} and \eqref{eq:derivation d}
respectively).
Although 
the fact that \eqref{senary relation} holds for any $\tilde f\in\dmr$
is  stated 
as a theorem of Ecalle in \cite[Theorem 3.1 and A.1]{S}, 
no proof seems to be provided in any references 
as far as the authors know.
Regarding the assumption \eqref{senary relation}, 
it should be  mentioned 
that the appendix of \cite{S} says the following:

\begin{prop}[\cite{S}]
\label{prop: Schneps Appendix}
Let $\tilde f$ be a Lie element in $\mathbb L$ with degree $>2$.
Put $M=\ma_{\tilde f}$.
Then giving the senary relation  \eqref{senary relation} 
for all $r\geqslant 1$ 
is equivalent to giving that
$d_F$ belongs to the Lie subalgebra $\sder$ of $\der$ 
(cf. Definition \ref{def:der and sder}).
\end{prop}

In \S \ref{sec:proof of theorem},
we extract a direct proof of this statement
from the arguments of \cite{S}.
As a corollary of  Theorem \ref{Schneps Theorem 1.1} and Proposition \ref{prop: Schneps Appendix}, the following inclusion is immediately obtained:

\begin{cor}[\cite{S}]
The map \eqref{eq:tilde f to F}
gives an inclusion
\begin{equation}\label{eq:dmr cap sder to krv}
\dmr\cap \nu^{-1}(\sder) \ \hookrightarrow \krv.
\end{equation}
\end{cor}

%
%

In this paper we check 
the validity of  the assumption  \eqref{senary relation}
in small $r$: 

\begin{thm}\label{thm:senary check}
The senary relation  \eqref{senary relation}
holds for any mould $M=\ma_{\tilde f}$ with $\tilde f\in\dmr$
when $r=1,2$ and $3$.
\end{thm}

The proof is given by mould theoretical arguments.

\begin{rem}
Since Drinfeld's Grothendieck-Teichm\"{u}ller Lie algebra $\grt$ is  contained in 
$\sder$ by \cite[Proposition 5.7]{Dr}
and the inclusion from $\grt$ to $\dmr$
is constructed in \cite{F11, EF2}, 
we have a sequence of inclusions
\begin{equation}
\grt \ \hookrightarrow \quad \dmr\cap\nu^{-1}(\sder) \quad \hookrightarrow \ \krv.
\end{equation}
\end{rem}

\begin{rem}
The inclusion 
$$
\Fil^2_{\mathcal D}\mathbb D(\Gamma)_{\bullet\bullet}\hookrightarrow
\mathfrak{lkrv}(\Gamma)_{\bullet\bullet},
$$
which is regarded to be 
a bigraded variant of \eqref{eq:dmr cap sder to krv} in the case of $\Gamma=\{e\}$,
where $\Gamma$ is an  abelian group, 
is constructed in \cite{FK}.
\end{rem}

The remainder of this paper proceeds as follows:
\S \ref{sec:mould theory} is a quick review on the notions of mould theory which are required in the paper.
We recall the definitions of two Lie algebras in  \S \ref{sec:2LAs}.
We give a stream-lined proof of Proposition \ref{prop: Schneps Appendix}
in \S \ref{sec:proof of theorem}.
Theorem \ref{thm:senary check}
is proved in \S \ref{sec:proof of senary check theorem}.

\subsubsection*{Acknowledgements}
H.F. has been supported by grants JSPS KAKENHI  JP18H01110, JP20H00115,
JP21H00969 and JP21H04430.
The authors would like to express their gratitude 
to L. Schneps and B. Enriquez for valuable comments and suggestions

\section{Mould theory}\label{sec:mould theory}
We briefly review several notions of  Ecalle's mould theory
by borrowing the notations  in our previous  joint paper \cite[\S 1]{FK} 
where moulds indexed by a finite abelian group $\Gamma$ is introduced.
We put $\mathcal{F}:=\bigcup_{m\geqslant 1}\mathbb Q[x_1,\dots,x_m].$



\begin{defn}[\cite{E81} I, pp.12-13 and \cite{S-ARIGARI}]
\label{def:mould}
A {\it mould} on $\Z_{\geqslant0}$ with values in $\mathcal{F}$ 
means  
a sequence
\begin{equation*}
	M=(M^m(x_1, \dots , x_m))_{m\in\Z_{\geqslant0}}=\bigl( M^0(\emptyset), M^1(x_1), M^2(x_1, x_2), \dots \bigr),
\end{equation*}
with $M^0(\emptyset)\in\mathbb Q$ and $M^m(x_1,\dots,x_m)$
$\in\mathbb Q[x_1,\dots,x_m]$ for $m\geqslant 1$. 
The set of all moulds with values in $\mathcal{F}$ is denoted by $\mathcal{M}(\mathcal{F})$. 
The set $\mathcal M(\mathcal F)$ 
forms a non-commutative, associative, unital $\mathbb Q$-algebra by
\begin{align*}
	M + N
	&:= (M^m(x_1, \dots , x_m)+ N^m(x_1, \dots , x_m))_{m\in\Z_{\geqslant0}}, \\
	c M
	&:= (c M^m(x_1, \dots , x_m))_{m\in\Z_{\geqslant0}}, \\
	 M \times N &
	:=\left(\sum_{i=0}^m M^i(x_1, \dots , x_i) N^{m-i}(x_{i+1}, \dots , x_m)\right)_{m\in\Z_{\geqslant0}}
\end{align*}
for $M, N\in\mathcal M(\mathcal F)$ and $c\in \mathbb Q$. 
Its unit $ I\in\mathcal M(\mathcal F)$ is given by 
$ I:=(1,0,0,\dots)$.

A mould $C=( C^m(x_1,\dots, x_m) )_{m\in\Z_{\geqslant0}}\in\mathcal M(\mathcal F)$ 
is called  \textit{constant} when we have $C^m(x_1,\dots, x_m)\in\Q$ for all $m\in\Z_{\geqslant0}$'s.
The unit $I$ is  a constant mould.
For our simplicity, we denote each component $C^m(x_1,\dots,x_m)$ of a constant mould $C$ by $C_m$.
\end{defn}

We equip $\mathcal M(\mathcal F)$ with the depth filtration
$\{\Fil^m_{\D}\mathcal M(\mathcal F)\}_{m\geqslant 0}$
where
$\Fil^m_{\D}\mathcal M(\mathcal F)$ consists of moulds
with $ M^r(x_1,\dots,x_r)=0$ for $r<m$.
The algebraic structure of $\mathcal M(\mathcal F)$ is compatible with the depth filtration and
the subspace
$$
\ARI
:=\{ M \in \mathcal M(\mathcal F) \ |\  M^0(\emptyset)=0 \}.
$$ 
forms a filtered (non-unital) subalgebra.

We call a mould $ M\in \mathcal M(\mathcal F)$ {\it finite} when
$ M^m(x_1, \dots , x_m)=0$ except for finitely many $m$'s.
We put the upper letters "$\fin$" 
as $\mathcal M(\mathcal F)^{\fin}$, $\ARI^{\fin}$, etc.
to indicate the finite part.

For our later use, we  prepare several operations on $\mathcal M(\mathcal F)$ :
\begin{defn}[{\cite{E-flex}}]
We consider the following $\Q$-linear operators 
defined by,
for any mould $ M=( M^m(u_1,\dots, u_m))_m$ in $\mathcal{M}(\mathcal{F})$,
\begin{align*}
&\swap( M)^m(v_1,\dots, v_m)= M^m(v_m,v_{m-1}-v_m,\dots,v_1-v_2), \\
&\pus( M)^m(u_1,\dots,u_m)= M^m(u_m,u_1,\dots,u_{m-1}),\\
&\teru( M)^m(u_1,\dots,u_m)= M^m(u_1,\dots,u_m) \\
&\qquad+\frac{1}{u_m}\left\{ M^{m-1}(u_1,\dots,u_{m-2},u_{m-1}+u_m)- M^{m-1}(u_1,\dots,u_{m-2},u_{m-1})\right\}.
\end{align*}
\end{defn}
We remark that $\swap\circ\swap=\id$. 

\begin{defn}
\label{def:ARIpushpusnu}
(1).
A mould $ M\in\mathcal{M}(\mathcal{F})$ 
is called
{\it pus-neutral} (\cite[(2.73)]{E-flex}) if
\begin{equation*}\label{pus-neutral}
	\sum_{i=1}^{m}\pus^i( M)^m(u_1,\dots,u_m)\left(=\sum_{i\in\Z/m\Z} M^m(u_{1+i},\dots,u_{m+i})\right)=0
\end{equation*}
holds for all $m\geqslant1$.

(2). The $\mathbb Q$-linear space
$\ARI_{\pspush/\pusnu}$ is
the subset of $\ARI$ consisting of
all moulds $ M$ 
satisfying Ecalle's senary relation \eqref{senary relation} 
for all $r\geqslant1$
and the pus-neutrality for $\swap(M)$.
\end{defn}

\bigskip 
To explain  the notion of the alternality of moulds,
we prepare some algebraic formulation:
Set $X:=\{x_i\}_{i\in\N}$.
We denote by $X_{\Z}$ the free $\Z$-module spanned by all elements of $X$.
Let $X_{\Z}^\bullet$ be
the non-commutative free monoid generated by all elements of $X_{\Z}$
with the empty word $\emptyset$ as the unit. 
We sometimes write each element $\omega=u_1 \cdots u_m \in X_{\Z}^\bullet$
with $u_1, \dots ,u_m\in X_{\Z}$ by
$\omega=(u_1, \dots ,u_m)$
as a sequence.
We put $l(\omega):=m$ and call this the {\it length} of $\omega=u_1 \cdots u_m$.

	For our simplicity, we  occasionally denote 
	\begin{equation*}
		 M=( M^m(\vecx_m)  )_{m\in \Z_{\geqslant0}} \quad \textrm{or}\quad
		 M=( M(\vecx_m) )_{m \in \Z_{\geqslant0}},
	\end{equation*}
	with $\vecx_0:=\emptyset$ and $\vecx_m:=(x_1, \dots ,x_m)$ for $m\geqslant1$. Under the notations,  the product of $ M, N\in\mathcal M(\mathcal F)$ is presented as
	\begin{equation*}
		 M\times N
	:=\left(
	\sum_{\substack{
		\vecx_m=\alpha\beta \\
		}}M^{l(\alpha)}(\alpha)N^{l(\beta)}(\beta)
	\right)_{m\in\Z_{\geqslant0}}
	\end{equation*}
	where $\alpha$ and $\beta$ run over  $X_{\Z}$.

We set
$\mathcal A:=\mathbb Q \langle X_{\Z} \rangle$,
the non-commutative polynomial algebra generated by 
$X_{\Z}$.
So $\mathcal A$ is the $\mathbb Q$-linear space generated by $X_{\Z}^{\bullet}$.
\footnote{
We should beware of the inequality $1\cdot (-x_1)\neq- (1\cdot x_1)$.
} 
We equip $\mathcal A$ with a product $\shuffle : \mathcal A\otimes \mathcal A \rightarrow \mathcal A$ 
which is the linear map defined by $\emptyset\, \shuffle\, \omega:=\omega\, \shuffle\, \emptyset:=w$ and
\begin{equation}
	u\omega\, \shuffle\, v\eta:=u(\omega\, \shuffle\, v\eta)+v(u\omega\, \shuffle\, \eta),
\end{equation}
for $u, v \in X_{\Z}$
and $\omega, \eta \in X_{\Z}^\bullet$.
Then the pair $(\mathcal A,\shuffle)$ is a commutative, associative, unital $\mathbb Q$-algebra.

We consider the family $\left\{\Sh{\omega}{\eta}{\alpha}\right\}_{\omega,\eta,\alpha\in X_{\Z}^\bullet}$ in $\Z$ defined by
\begin{equation*}
	\omega\ \shuffle\ \eta
	=\sum_{\alpha\in X_{\Z}^\bullet}
	\Sh{\omega}{\eta}{\alpha}\alpha.
\end{equation*}
In particular, for $p,q\in\N$ and $u_1,\dots,u_{p+q}\in X_{\Z}$, 
the shuffle product is rewritten as
\begin{equation*}
	(u_1,\dots,u_p)\ \shuffle\ (u_{p+1},\dots,u_{p+q})=\sum_{\sigma\in\Sha_{p,q}}(u_{\sigma(1)},\dots,u_{\sigma(p)},u_{\sigma(p+1)},\dots,u_{\sigma(p+q)}),
\end{equation*}
where the set $\Sha_{p,q}$ stands for 
\begin{equation}\label{shuffle permutation}
	\{\sigma\in S_{p+q}
	\ |\ \sigma^{-1}(1)<\cdots<\sigma^{-1}(p),\ \sigma^{-1}(p+1)<\cdots<\sigma^{-1}(p+q)\},
\end{equation}
and $S_{p+q}$ means the symmetry group with degree $p+q$.

Let us assume that $u_1,\dots,u_m\in X_{\Z}$ are algebraically independent over $\mathbb Q$ in $\mathcal F$.
For $ M\in\mathcal M(\mathcal F)$,
$		 M^m(u_1,\dots,u_m)$
stands for the image of $M^m(x_1,\dots, x_m)$ under the embedding
$\mathbb Q[x_1,\dots,x_m]\hookrightarrow \mathcal F$
which sends $x_i\mapsto u_i$.

\begin{defn}[\cite{E81} I--p.118]
A mould $ M$ is  {\it alternal} when $ M(\emptyset)=0$ and
\begin{equation}\label{eq:alternal}
\sum_{\alpha\in X_{\Z}^\bullet}\Sh{(x_1,\dots,x_p)}{(x_{p+1},\dots,x_{p+q})}{\alpha} M^{p+q}(\alpha)=0,
\end{equation}
holds for all $p,q\geqslant1$.
\end{defn}


\begin{defn}[\cite{E-ARIGARI,E-flex}]\label{def:ARIal}
The $\mathbb Q$-linear space 
$\ARI_\al$
is defined to be the subset of all alternal moulds in $\ARI$.
\end{defn}
 It  forms a filtered $\mathbb Q$-algebra under the filtrations induced from the depth filtration of $\ARI$.

\begin{rem}\label{rem: ari-bracket}
The  {\it ari-bracket} is introduced 
in \cite[\S 2.2]{E-flex} (cf. \cite[\S 2.2]{S-ARIGARI}) and
it is known that
$\ARI$ and $\ARI_\al$ 
form  Lie algebras 
under the bracket
 (see \cite{FK} for a full detailed proof).
\end{rem}

\bigskip 
To present  the notion of the alternility of moulds,
we prepare some algebraic formulation:
Set $Y:=\{y_i\}_{i\in\N}$.  Let $\mathcal Y:=\mathbb Q(Y)$ be the commutative field generated by all elements of $Y$ over $\mathbb Q$. 
We denote $\mathcal A_{\mathcal Y}:=\mathcal Y\langle X_{\Z}\rangle$ to be the non-commutative polynomial algebra generated by all elements of
$X_{\Z}$.
So $\mathcal A_{\mathcal Y}$ is the same as the tensor product $\mathcal A\otimes_{\mathbb Q} \mathcal Y$
with $\mathcal A=\mathbb Q\langle X_{\Z}\rangle$. 
It is  equipped with a product $\hshuffle:\mathcal A_{\mathcal Y}^{\otimes 2}\rightarrow\mathcal A_{\mathcal Y}$ which is the $\mathcal Y$-linear map defined by $\emptyset\ \hshuffle\ \omega:=\omega\ \hshuffle\ \emptyset:=\omega$ and 
\begin{equation}\label{shuffle star definition}
	u\omega\ \hshuffle\ v\eta
	:=u(\omega\ \hshuffle\ v\eta)+v(u\omega\ \hshuffle\ \eta)
	+f(u-v)\{u(\omega\hshuffle\eta)-v(\omega\hshuffle\eta)\},
\end{equation}
for $u,v\in X_{\Z}$
and $\omega,\eta\in X_{\Z}^\bullet$. Here, we define the above map $f:X_{\Z}\rightarrow\mathcal Y$ by $f(0):=0$ and
\begin{equation*}
	f(n_1x_{i_1}+\cdots+n_kx_{i_k}):=\frac{1}{n_1y_{i_1}+\cdots+n_ky_{i_k}},
\end{equation*}
for $k\geqslant1$ and for $i_1,\dots,i_k\in\N$ and for $n_1,\dots,n_k\in\Z$ with $n_j\neq0$ for a certain $1\leqslant j\leqslant k$. Then the pair $(\mathcal A_{\mathcal Y},\hshuffle)$ is a commutative, {\it non}-associative, unital $\mathcal Y$-algebra.

We consider the family $\left\{ \hSh{\omega}{\eta}{\alpha} \right\}_{\omega,\eta,\alpha\in X_{\Z}^\bullet}$ in $\mathcal Y$ defined by
\begin{equation*}
	\omega\ \hshuffle\ \eta
	=\sum_{\alpha\in X_{\Z}^\bullet}
	\hSh{\omega}{\eta}{\alpha}\alpha.
\end{equation*}

\begin{defn}[\cite{E-ARIGARI,E-flex, SS}]
\label{ARIalil}
(1). A mould $ M$ is {\it alternil} when $ M(\emptyset)=0$ and
\begin{equation}\label{eq:alternil}
\sum_{\alpha\in X_{\Z}^\bullet}\left.\hSh{(x_1,\dots,x_p)}{(x_{p+1},\dots,x_{p+q})}{\alpha}\right|_{y_i=x_i} M^{l(\alpha)}(\alpha)=0,
\end{equation}
for all $p,q\geqslant1$.

(2). 
The subset $\ARI_{\underline\al\ast\underline\il}$ 
is the $\Q$-linear space spanned by all alternal
moulds $M$ in $\ARI$ such that  $\swap(M)$ is alternil up to a constant-valued mould,
(that is, there exist a constant mould $C\in\mathcal M(\mathcal F)$ so that $\swap(M)+C$ is alternil).
%
\end{defn}

\begin{rem}\label{rem:SaSch}
Under the $\ari$-bracket (cf. Remark \ref{rem: ari-bracket}),
the $\mathbb Q$-linear space 
$\ARI_{\underline\al\ast\underline\il}$
forms a filtered Lie algebra 
(cf. 
\cite{K, S-ARIGARI,SS}).
The Lie algebra isomorphism between $\dmr$ and $\ARI_{\underline\al\ast\underline\il}$
is realized by the map sending $\tilde f\mapsto \ma_{\tilde f}$
(cf. \cite{SS, S-ARIGARI}).
\end{rem}

Let $\mathbb A=\mathbb Q\langle x,y\rangle$  be
the non-commutative polynomial algebra with two variables $x$ and $y$.
This $\mathbb A$ forms a bigraded algebra;
$\mathbb A=\oplus_{w,d}\mathbb A_{w,d}$,
where $\mathbb A_{w,d}$ means the $\mathbb Q$-linear space generated by all monomials with {\it weight} (the total degree) $w$ and {\it depth} (the degree with respect to $y$) $d$.
Put  $\mathbb A_w=\oplus_d\mathbb A_{w,d}$ and
$\Fil_{\D}^d \mathbb A_w:=\oplus_{N\geqslant d}\mathbb A_{w,N}$
for $d>0$.
Then $\mathbb A$ is a filtered  graded algebra
by the depth filtration $\Fil_{\D}^d \mathbb A=\oplus_w\Fil_{\D}^d \mathbb A_w$.

\begin{defn}[{\cite[Appendix A]{S}}]
Let $h\in\mathbb A_w$ be a degree $w$ homogeneous polynomial
with
$$
h=\sum_{r=0}^wh^r \qquad \text{and}\qquad 
h^r= \sum_{\mathbf{e}=(e_0,\dots,e_r)\in\mathcal{E}_w^r} a(h)_{e_0,\dots,e_r}x^{e_0}y\cdots yx^{e_r}\in
\mathbb A_{w,r}
$$
where $\mathcal{E}_w^r=\{\mathbf{e}=(e_0,\dots,e_r)\in\N_0^{r+1} \mid \sum_{i=0}^r e_i=w-r \}$.
With such $h\in\mathbb A$,
the associated  mould 
$$\ma_h=(\ma^0_h,\ma^1_h,\ma^2_h,\dots,\ma^w_h,0,0,\dots) \in\mathcal M(\mathcal F)$$
is defined by $\ma^0_h=0$ and 
\begin{align*}
&\ma^r_h=\ma^r_h(u_1,\dots,u_r)
=\vimo^r_h(0,u_1,u_1+u_2,\dots,u_1+\cdots+u_r), \\
&\vimo^r_h(z_0,\dots,z_r)
=\sum_{(e_0,\dots,e_r)\in\mathcal{E}_w^r}
a(h)_{e_0,\dots,e_r}
z_0^{e_0}z_1^{e_1}z_2^{e_2}\cdots z_r^{e_r}.
\end{align*}
\end{defn}

The map $h\mapsto\ma_h$
yields a filtered algebra homomorphism 
$$\ma:\mathbb A\to \mathcal {M}(\mathcal F)^\fin . $$
Various properties of the map are investigated in \cite{FHK}.

\section{Kashiwara-Vergne and double shuffle Lie algebra}\label{sec:2LAs}
We recall the definition of 
Kashiwara-Vergne Lie algebra $\krv$
introduced in \cite{AET} and \cite{AT}
and also that of the double shuffle Lie algebra  $\dmr$
introduced in \cite{R}.
Then we explain their mould theoretic interpretations.

Let $\mathbb L=\oplus_{w\geqslant 1}\mathbb L_w
=\oplus_{w,d}\mathbb L_{w,d}$ 
be the subspace of $\mathbb A$ consisting of all Lie polynomials.
It is
a free graded Lie $\mathbb Q$-algebra with two variables $x$ and $y$ with $\deg x=\deg y=1$.
We equip $\mathbb L$ with  
the depth filtration
$\Fil_{\D}^d \mathbb L=\oplus_w\Fil_{\D}^d \mathbb L_w$ by
$\Fil_{\D}^d \mathbb L_w:=\oplus_{N\geqslant d}\mathbb L_{w,N}$ for $d>0$.


\begin{defn}\label{def:der and sder}
(1).
We denote by $\der$ the set of derivations of $\mathbb L$ and
by $\tder$ its subset of {\it tangential derivations} which are
the derivations  $D_{F,G}$ of $\mathbb L$
such that $D_{F,G}(x)= [x,G]$ and $D_{F,G}(y)= [y,F]$ for some $F,G\in \mathbb L$.
It forms a Lie algebra under the bracket given by
\begin{equation*}\label{bracket}
[D_{F_1,G_1},D_{F_2,G_2}]
=D_{F_1,G_1}\circ D_{F_2,G_2}-D_{F_2,G_2}\circ D_{F_1,G_1}.
\end{equation*}

(2).
We define by $\sder$ the set of {\it special derivations},
which are tangential derivations such that 
\begin{equation}\label{eq:special derivation}
D(z)=0
\quad\text{  with}\quad z:=-x-y.
\end{equation}

\end{defn}

The subspace $\sder$ 
is a Lie subalgebra of $\tder$ and there is a sequence of Lie algebras
$$
\sder\subset \tder \subset \der.
$$

\begin{defn}[\cite{AT,AET}]\label{defn:krv}
The {\it Kashiwara-Vergne Lie algebra} is the graded $\mathbb Q$-linear space $\krv=\oplus_{w> 1}\krv_w$.
\footnote{
For our convenience we neglect the degree 1 part,
although it is taken in account in \cite{AT}.
}
Here
its degree $w$-part $\krv_w$  consists of Lie elements  $F\in\mathbb L_w$
satisfying
\begin{equation}\tag{KV1}\label{KV1}
[x,G]+[y,F]=0
\end{equation}
\begin{equation}\tag{KV2}\label{KV2}
\tr (G_x x+F_y y)=\alpha\cdot \tr\left((x+y)^w-x^w-y^w\right)
\end{equation}
with some  $G=G(F)\in\mathbb L_w$ and $\alpha\in\mathbb Q$, 
when we write $F=F_xx+F_yy$ and $G=G_xx+G_yy$ in $\mathbb A$.
Here $\tr$ is the trace map, the natural projection  from $\mathbb A$ 
to the $\mathbb Q$-linear space $\mathrm{Cyc}(\mathbb A)$ of cyclic words
(cf. \cite{AT}).
\end{defn}

We have $\krv_2=\{0\}$.
We note that such $G=G(F)$ is  unique if it exists when $w>1$
(for the calculations, see \cite[Theorem 1.1]{S}).
The condition \eqref{KV1} is equivalent to $D_{F,G(F)}(x+y)=0$,
i.e. $D_{F,G(F)}\in\sder$.
The Lie algebra structure of $\krv$ is defined to
make the embedding  
$$
\krv \subset \sder
$$
sending $F\mapsto D_{F, G(F)}$
a Lie algebra homomorphism.

For $d\geqslant 1$, we denote by $\Fil_{\D}^d\krv_w$ the subspace of $\krv_w$ 
consisting of $F\in \Fil_{\D}^d\mathbb L_w$ and put
$\Fil_{\D}^d\krv=\oplus_w\Fil_{\D}^d\krv_w$.
Then $\krv$ forms a filtered Lie algebra with the filtration 
$\{ \Fil_{\D}^d\krv\}_d$.

The following is a mould theoretical interpretation of $\krv$:

\begin{thm}[\cite{FK,RS}]
The map $F\in\mathbb A\mapsto\ma_{\tilde f}\in\mathcal M({\mathcal F})$ gives an isomorphism 
of filtered $\mathbb Q$-linear spaces
$$\Fil_{\D}^{2}\krv\simeq
\ARI_{\pspush/\pusnu}\cap\ARI_\al^{\fin},$$
where
we put 
\begin{equation}\label{eq:F f tilde f}
f(x,y)=F(z,y) \quad \text{and}\quad \tilde f(x,y)=f(x,-y)
\end{equation}
for $F\in\mathbb L$.
\end{thm}

For $\ARI_{\pspush/\pusnu}$ and $\ARI_\al$,
 see Definition \ref{def:ARIpushpusnu}
and \ref{def:ARIal} respectively.

\begin{rem}
The above theorem is extended to $\krv(\Gamma)$
with an arbitrary abelian group $\Gamma$
and a mould theoretic interpretation of
its bigraded variant $\lkrv(\Gamma)$ is provided
in \cite{FK}.
\end{rem}

\bigskip

Let $\mathbb A_Y:=\Q\langle y_1,y_2,y_3,\dots\rangle$ be the free associative  non-commutative $\Q$-algebra generated 
by the variables $y_m$'s ($m\in\N$)
with unit.
It is equipped with a structure of Hopf algebra with the coproduct 
$\Delta_\ast$ given by $\Delta_\ast(y_n)=\sum_{i=0}^n y_i\otimes y_{n-i}$
($y_0:=1$). 
Let $\pi_Y:\mathbb A\to\mathbb A_Y$ be
the $\Q$-linear projection
that sends all the words ending in $x$ to zero and the
word $x^{n_m-1}y\cdots x^{n_1-1}y$ ($n_1,\dots,n_m\in\N$)
to 
$y_{n_m}\cdots y_{n_1}$.
For $\varphi=\sum_{W:\text{word}} c_W(\varphi) W\in\mathbb A$,
define 
\begin{equation*}
\varphi_*=\varphi_{\text{corr}}+\pi_Y(\varphi)
\quad\text{with}\quad 
\varphi_{\text{corr}}=\sum_{n=1}^{\infty}
\frac{(-1)^{n-1}}{n}c_{x^{n-1}y}(\varphi)y_1^n.
\end{equation*}

\begin{defn}[\cite{R}]\label{defn:dmr}
Set-theoretically the {\it double shuffle Lie algebra}
$\dmr=\oplus_{w> 1}\dmr_w$ is defined to be the set 
of $\tilde f\in\mathbb L$ 
satisfying $c_{xy}(\tilde f)=0$ and
\begin{equation}\label{Lie double shuffle}
\Delta_*(\tilde f_*)=1\otimes \tilde f_*+\tilde f_*\otimes 1.
\end{equation}
\end{defn}

For $\varphi\in\mathbb L$, let $d_\varphi\in\der$ be the derivation determined by 
\begin{equation}\label{eq:derivation d}
d_\varphi(y)=[y,\varphi(x,y)] 
\quad\text{and}\quad 
d_\varphi(x+y)=0.
\end{equation}
Then \cite[Proposition 4.A.i]{R} claims that
the inclusion
$$
\dmr\subset\der
$$
sending ${\tilde f}\mapsto d_{F}$
equips $\dmr$  with a structure of Lie subalgebra  of $\der$.
For $d\geqslant 1$, we denote by  $\Fil_{\D}^d\dmr_w$ the subspace of $\dmr_w$ 
consisting of $\tilde f\in \Fil_{\D}^d\mathbb L_w$
and we put $\Fil_{\D}^d\dmr=\oplus_w\Fil_{\D}^d\dmr_w$.
Then $\dmr$ forms a filtered Lie algebra with the filtration 
$\{ \Fil_{\D}^d\dmr\}_d$.


A mould theoretical interpretation of $\dmr$ is given as follows:

\begin{thm}
[{\cite[Theorem 3.4.4]{S-ARIGARI}}]
The map  $\tilde f\in\mathbb A	\mapsto\ma_{\tilde f}\in\mathcal M({\mathcal F})$ gives an isomorphism 
of filtered Lie algebras
$$\Fil_{\D}^{2}\dmr\simeq
\Fil_{\D}^{2}\ARI_{\underline\al\ast\underline\il}^{\fin}.$$
\end{thm}

For $\ARI_{\underline\al\ast\underline\il}$, see Definition \ref{ARIalil}.

\begin{rem}
A mould theoretic interpretation of
Goncharov's dihedral Lie algebra ${\mathbb D}(\Gamma)$ (cf. \cite{G}),
which is regarded to be a bigraded variant of  $\dmr$ in the case of $\Gamma=\{e\}$,
is given in \cite{FK}.
\end{rem}

\section{The senary relation and $\sder$}
\label{sec:proof of theorem}

 

For the reader's convenience,
we give a streamlined proof of Proposition \ref{prop: Schneps Appendix}
by extracting the arguments in \cite{S} and also
reproducing those of \cite[\S 2]{FK} 
in the case of $\Gamma=\{e\}$.

An element $g\in\mathbb A_w$  is called {\it anti-palindromic} when $g=(-1)^w\anti(g)$
where $\anti$ is the palindrome (backwards-writing) operator (explained in \cite[Definition 1.3]{S}).


\begin{lem}[{\cite[Theorem 2.1]{S}}]\label{lem:refo1:KV1}
Let $F\in\mathbb L_w$ with $w>2$. 
Then  \eqref{KV1} for $F$ is equivalent to 
the anti-palindrome  for
$\tilde f_y+\tilde f_x$ 
when we write $\tilde f=\tilde f_xx+\tilde f_yy$.
\end{lem}


The claim is nothing but the equivalence between (i) and (v) of {\cite[Theorem 2.1]{S}},
which is deduced from their equivalence with the intermediate conditions (ii), (iii), (iv) in loc.~cit.~
We  present a shorter proof below.

\begin{proof}
Let us write $F=F_xx+F_yy=xF^x+yF^y$.
By \cite[Proposition 2.6]{S}, we have
$$
f_y-f_x=F_y(-x-y,y).
$$
Therefore
the anti-palindrome for $\tilde f_y+\tilde f_x$,
so for $f_y-f_x$, 
is equivalent to  the anti-palindrome for $F_y$.
Since we have $F^y=(-1)^{w-1}\anti(F_y)$ for $F\in\mathbb L$,
it is equivalent to $$F_y=F^y.$$
Therefore our claim is reduced to prove an equivalence  between
 \eqref{KV1} 
 and $F_y=F^y$.

Assume \eqref{KV1} for $F$.
Set $H=[y,F]$.
We have
\begin{equation*}\label{eqn for H}
H=yF_yy+yF_xx-yF^yy-xF^xy. 
\end{equation*}
We have $H=Gx-xG$ by $[y,F]+[x,G]=0$. 
So $H$ has no words starting and ending in $y$.
Thus we must have $F_y=F^y$.

Conversely assume  $F_y=F^y$.
Then  $H$ has no words starting and ending in $y$.
By \cite[Proposition 2.2]{S}, there is a $G\in\mathbb L_{w}$
such that $H=[G,x]$,
which is calculated to be $G(F)=s'(F_x)$.
We get \eqref{KV1}.
%
%

%
Whence we obtain the claim.
\end{proof}

The following is 
indebted of
the arguments in \cite[Appendix A]{S}.

\begin{lem}\label{lem:refo2:KV1}
Let $\tilde f\in\mathbb L_w$ with $w>2$. 
Then anti-palindrome for $\tilde f_y+\tilde f_x$
is equivalent to the senary relation \eqref{senary relation}
for $1\leqslant r \leqslant w$
with $ M=\ma_{\tilde f}$.
\end{lem}

\begin{proof}
In this proof, we quote the same equation numbers in \cite[Appendix A]{S}.

The first two paragraphs of
the proof of \cite[Proposition A.3]{S} says that
the anti-palindrome for $\tilde f_y+\tilde f_x$ 
is equivalent to 
\begin{equation*}
\vimo_{\tilde f_y^r+\tilde f_x^r}^r  (z_0,\dots,z_r) 
= (-1)^{w-1}
\vimo_{\anti(\tilde f_y^r+\tilde f_x^r)}^r  (z_0,\dots,z_r)
\end{equation*}
for $0\leqslant r \leqslant w-1$. 
In the proof of \cite[Proposition A.3]{S}, 
it is calculated that
its left hand side is 
\begin{align}\tag{A19}\label{A19}
&\vimo_{\tilde f^{r+1}}^{r+1}  (z_0,\dots,z_r,0) \\ \notag
&+\frac{1}{z_r}\left\{
\vimo^m_{\tilde f^r}(z_0,\dots,z_{r-1},z_r)-\vimo^r_{\tilde f^r}(z_0,\dots,z_{r-1},0)
\right\}
\end{align}
and while its right hand side is calculated to be
\begin{align}\tag{A20}\label{A20}
(-1)^{w-1}\bigl[
&\vimo_{\tilde f^{r+1}}^{r+1}  (z_r,\dots,z_0,0) \\ \notag
+&\frac{1}{z_0}\left\{
\vimo^r_{\tilde f^r}(z_r,\dots,z_{1},z_0)-\vimo^r_{\tilde f^r}(z_r,\dots,z_{1},0)
\right\}
\bigr].
\end{align}
It is noted  that the first term  corresponds to  $\tilde f^r_y$ and the second  one corresponds to $\tilde f^r_x$ in both equations.
Whence \eqref{KV1} is equivalent to 
$$\eqref{A19}=\eqref{A20}$$
for $1\leqslant r \leqslant w-1$.

On the other hand,
the  senary relation \eqref{senary relation} is equivalent to
\begin{equation}\tag{A8}\label{A8}
\swap\circ\teru( M)^r
=\swap\circ\push\circ\mantar\circ\teru\circ\mantar( M)^r.
\end{equation}
By definition, its left hand side is calculated to be
\begin{align}\tag{A10}\label{A10}
&\vimo_{\tilde f^r}^r(0,v_r,\dots,v_1) \\ \notag
+\frac{1}{v_1-v_2}&\{
\vimo_{\tilde f^{r-1}}^{r-1}(0,v_r,\dots,v_3,v_1)-\vimo_{\tilde f^{r-1}}^{r-1}
(0,v_r,\dots,v_3,v_2)\}.
\end{align}
While by 
$\vimo_{\tilde f^r}^r(z_0,\dots,z_r) =(-1)^{w-r}\vimo_{\tilde f^r}^r(-z_0,\dots,-z_r) $
and
\begin{equation}\tag{A13}\label{A13}
\vimo_{\tilde f^r}^r(z_0,\dots,z_r) =\vimo_{\tilde f^r}^r(0,z_1-z_0,\dots,z_r-z_0)
\end{equation}
for $\tilde f\in\mathbb L$, 
its right hand side is 
\begin{align}\tag{A14}\label{A14}
&(-1)^{w-1} \Bigl[ 
\vimo_{\tilde f^r}^r  (v_2,v_3,\dots,v_r,0,v_1) \\ \notag
\qquad+&\frac{1}{v_1}\left\{
\vimo_{\tilde f^{r-1}}^{r-1}(v_2,v_3,\dots,v_r,v_1)-\vimo_{\tilde f^{r-1}}^{r-1}(v_2,v_3,\dots,v_r,0)\right\}
\Bigr].
\end{align}
Therefore the senary relation \eqref{senary relation} is equivalent to 
$$\eqref{A10}=\eqref{A14}$$
for $1\leqslant r \leqslant w$.

By \eqref{A13} and the change of variables 
$z_0=-v_1, z_1=v_r-v_1,\dots, z_{r-1}=v_2-v_1$,
the equation \eqref{A10} is calculated to be
\begin{align*}
&\vimo_{\tilde f^r}^r(z_0,\dots,z_{r-1},0) \\
&+\frac{1}{z_{r-1}}
\left\{
\vimo^{r-1}_{\tilde f^{r-1}}(z_0,\dots,z_{r-2},z_{r-1})-\vimo^{r-1}_{\tilde f^{r-1}}(z_0,\dots,z_{r-2},0)
\right\},
\end{align*}
which is  \eqref{A19} with $r$ replaced with $r-1$. 
Again by \eqref{A13} and the above change of variables, 
the equation \eqref{A14} is calculated to be
\begin{align*}
&(-1)^{w-1}\Bigl[
\vimo_{\tilde f^r}^r(z_{r-1},\dots,z_1,z_{0},0) \\
& \qquad\qquad +\frac{1}{z_0}
\left\{
\vimo_{\tilde f^{r-1}}^{r-1}(z_{r-1},\dots,z_1,z_0) 
-\vimo_{\tilde f^{r-1}}^{r-1}(z_{r-1},\dots,z_1,0) 
\right\}
\Bigr],
\end{align*}
which is \eqref{A20}  with $r$ replaced with $r-1$.
Whence we obtain the equivalence between \eqref{KV1} and
\eqref{senary relation}.
\end{proof}

\bigskip
Let $F$ (so whence $\tilde f$) be in $\mathbb L_w$ with $w>2$.
By definition, saying $d_F\in\sder$ is equivalent to
\eqref{KV1} for $F$.
While by Lemmas \ref{lem:refo1:KV1} and \ref{lem:refo2:KV1},
it is equivalent to 
the senary relation \eqref{senary relation} for $M=\ma_{\tilde f}$.
Therefore Proposition \ref{prop: Schneps Appendix} follows.


\section{The senary relation for small $r$}
\label{sec:proof of senary check theorem}
We check the validity of the assumption \eqref{senary relation} 
for any mould $M\in\Fil_{\D}^{2}\ARI_{\underline\al\ast\underline\il}$
in small depth $r=1,2,3$.

We define the parallel translation map $\h:\mathcal M(\mathcal F) \to \mathcal M(\mathcal F)$ by
$$
\h(M)^m(\vecx_m)
:=\left\{\begin{array}{ll}
M^m(\vecx_m) & (m=0,1), \\
M^{m-1}(x_2-x_1,\dots,x_m-x_1) & (m\geqslant2),
\end{array}\right.
$$
for $M\in\mathcal M(\mathcal F)$, and for our simplicity, we put
$$
\tswap=\h\circ\swap.
$$
For $m\geqslant2$ and for $1\leqslant i\leqslant m-1$, we consider
the collision map
$$\coll^m_{{i,i+1}}:\mathcal M(\mathcal F)\rightarrow\mathcal M(\mathcal F)$$
which is  $\mathbb Q$-linearly defined by the collision map
$\coll^m_{{i,i+1}}(M)^j(\vecx_j):=(M)^j(\vecx_j)$
 when $j\neq m$ and
\begin{align*}
	&\coll^m_{{i,i+1}}(M)^m(\vecx_m) :=\frac{1}{x_i-x_{i+1}}
	\left\{ M^{m-1}(x_1,\dots,x_{i-1},x_i,x_{i+2},\dots,x_m) \right. \\
	&\hspace{5cm}\left. -M^{m-1}(x_1,\dots,x_{i-1},x_{i+1},x_{i+2},\dots,x_m) \right\}
\end{align*}
for $M\in\mathcal M(\mathcal F)$.

\begin{lem}\label{lem:senary relation 2}
Let $r\geqslant1$. Ecalle's senary relation \eqref{senary relation} for  $M\in\ARI$ is equivalent to
	\begin{align}\label{senary relation 2}
		&\tswap(M)^{r+1}(x_1,\dots,x_{r+1})+\coll^{r+1}_{{2,3}}\circ \tswap(M)^{r+1}(x_1,\dots,x_{r+1}) \\
		&= \tswap(M)^{r+1}(x_2,\dots,x_{r+1},x_1)+\coll^{r+1}_{{1,2}}\circ \tswap(M)^{r+1}(x_1,\dots,x_{r+1}). \nonumber
	\end{align}
\end{lem}
\begin{proof}
	We expand both sides of \eqref{senary relation 2}. By the definition of two maps $\coll^{r+1}_{{2,3}}$ and $u$, the left hand side of \eqref{senary relation 2} is as follows:
	\begin{align*}
		 \tswap(M)^{r+1}&(x_1,\dots,x_{r+1})+\coll^{r+1}_{{2,3}}\circ \tswap(M)^{r+1}(x_1,\dots,x_{r+1}) \\
		=&M^r(x_{r+1}-x_1,x_r-x_{r+1},\dots,x_2-x_3) \nonumber \\
		&+\frac{1}{x_2-x_3}
			\left\{ M^{r-1}(x_{r+1}-x_1,x_r-x_{r+1},\dots,x_4-x_5,x_2-x_4) \right. \nonumber \\
		&\hspace{2cm}\left. -M^{r-1}(x_{r+1}-x_1,x_r-x_{r+1},\dots,x_4-x_5,x_3-x_4) \right\}. \nonumber
	\end{align*}
	On the other hand, the right hand side of \eqref{senary relation 2} is as follows:
	\begin{align*}
	 \tswap(M)^{r+1}&(x_2,\dots,x_{r+1},x_1)+\coll^{r+1}_{{1,2}}\circ \tswap(M)^{r+1}(x_1,\dots,x_{r+1}) \\
		=&M^r(x_1-x_2,x_{r+1}-x_1,x_r-x_{r+1},\dots,x_3-x_4) \nonumber \\
		&+\frac{1}{x_2-x_1}
			\left\{ M^{r-1}(x_{r+1}-x_2,x_r-x_{r+1},\dots,x_4-x_5,x_3-x_4) \right. \nonumber \\
		&\hspace{2cm}\left. -M^{r-1}(x_{r+1}-x_1,x_r-x_{r+1},\dots,x_4-x_5,x_3-x_4) \right\}. \nonumber
	\end{align*}
	So by putting $y_1:=x_{r+1}-x_1$ and $y_i:=x_{r+2-i}-x_{r+3-i}$ ($2\leqslant i\leqslant r$), we get
	\begin{align*}
		&M^r(y_1,y_2,\dots,y_r)
		+\frac{1}{y_r}
			\left\{ M^{r-1}(y_1,\dots,y_{r-2},y_{r-1}+y_r) 
			-M^{r-1}(y_1,\dots,y_{r-1}) \right\} \\
		&=M^r(-y_1-\dots-y_r,y_1,\dots,y_{r-1}) \nonumber \\
		&+\frac{1}{y_1+\dots+y_r}
			\left\{ M^{r-1}(-y_2-\dots-y_r,y_2,\dots,y_{r-1})  -M^{r-1}(y_1,\dots,y_{r-1}) \right\}.
	\end{align*}
	This is just equal to
	\begin{equation*}
		\teru( M)^r(y_1,y_2,\dots,y_r)
		=\push\circ\mantar\circ\teru\circ\mantar( M)^r(y_1,y_2,\dots,y_r).
	\end{equation*}
	Therefore, we obtain the claim.
\end{proof}

\begin{prop}\label{appendix theorem}
	Suppose that 
	 {$M\in\Fil_{\D}^{2}\ARI_{\underline\al\ast\underline\il}$}.
	Then the mould $M$ 
	satisfies Ecalle's senary relation 
	for $r=1$, $2$ and $3$.
\end{prop}

\begin{proof}
	Since we have $M^r(\vecx_r)=0$ for $r=0,1$ by $M\in\Fil_{\D}^{2}\ARI$, it is clear that $M$ satisfies the senary relation for $r=1$. By Lemma \ref{lem:senary relation 2}, it is enough to prove \eqref{senary relation 2} for $r=2,3$.
	
In the case when $r=2$,
by $M\in\ARI_{\underline\al\ast\underline\il}$, there exists a constant mould $C=(C_m)_{m\in\Z_{\geqslant0}}\in\mathcal M(\mathcal F)$ such that 
	$$\sum_{\alpha\in X_{\Z}^\bullet}
		\hSh{(x_2)}{(x_3)}{\alpha}
		\left( \tswap(M)^{l(\alpha)+1}(x_1,\alpha) + \tswap(C)^{l(\alpha)+1}(x_1,\alpha) \right)=0,$$
	that is, we have
	\begin{equation}\label{mantar invariant of depth=2}
		\tswap(M)^3(x_1,x_2,x_3)=-\tswap(M)^3(x_1,x_3,x_2)-2C_2.
	\end{equation}
	Here we use $\tswap(C)^3(x_1,x_2,x_3)=\tswap(C)^3(x_1,x_3,x_2)=C_2$.
	So by using \cite[Lemma 2.5.3]{S-ARIGARI} (see also \cite[Lemma 3.7]{FK}) and \eqref{mantar invariant of depth=2}, we have
	\begin{align}\label{il G70}
		\tswap & (M)^{3}(x_1,x_2,x_3)
		=\swap(M)^{2}(x_2-x_1,x_3-x_1) 
		=M^{2}(x_3-x_1,x_2-x_3)  \\ 
		&=-M^{2}(x_2-x_3,x_3-x_1) 
		=-\swap(M)^{2}(x_2-x_1,x_2-x_3) \nonumber \\
		&= -\tswap(M)^3(-x_2,-x_1,-x_3) 
		=\tswap(M)^{3}(-x_2,-x_3,-x_1) +2C_2. \nonumber
	\end{align}
	By using \eqref{il G70} twice, we get
	\begin{align*}
		\tswap & (M)^{3}(x_1,x_2,x_3)
		=\tswap(M)^{3}(x_3,x_1,x_2) +4C_2.
	\end{align*}
	By using this three times, we obtain
	\begin{align*}
		\tswap & (M)^{3}(x_1,x_2,x_3)
		=\tswap(M)^{3}(x_1,x_2,x_3) +12C_2,
	\end{align*}
	that is, we have $C_2=0$.
	So we get
	\begin{align}\label{eqn:il G70,71,72}
		\tswap(M)^3(x_1,x_2,x_3) = \tswap(M)^3(x_2,x_3,x_1).
	\end{align}
	Because $M^1(\vecx_1)=0$, we have $\coll^3_{{2,3}}\circ \tswap(M)^3(x_1,x_2,x_3)=\coll^3_{{1,2}}\circ \tswap(M)^3(x_1,x_2,x_3)=0$. Hence, we obtain \eqref{senary relation 2} for $r=2$.
We note that, by \eqref{il G70}, we have
\begin{align}\label{eqn:coll34 from il G70}
&\coll^4_{{3,4}}\circ \tswap(M)^4(x_1,x_3,x_2,x_4)
=-\coll^4_{{2,3}}\circ \tswap(M)^4(-x_3,-x_2,-x_4,-x_1), \\
\label{eqn:coll12 from il G70}
&\coll^4_{{1,2}}\circ \tswap(M)^4(x_1,x_2,x_3,x_4)
=-\coll^4_{{3,4}}\circ \tswap(M)^4(-x_3,-x_4,-x_1,-x_2).
\end{align}	

In the case when $r=3$, by $M\in\ARI_{\underline\al\ast\underline\il}$,
we have
\begin{align*}
0&=\sum_{\alpha\in X_{\Z}^\bullet}
	\hSh{(-x_2)}{(-x_4,-x_3)}{\alpha}
	\left( \tswap(M)^{l(\alpha)+1}(-x_1,\alpha) + \tswap(C)^{l(\alpha)+1}(-x_1,\alpha) \right) \\
&\quad -\sum_{\alpha\in X_{\Z}^\bullet}
	\hSh{(-x_2,-x_3)}{(-x_4)}{\alpha}
	\left( \tswap(M)^{l(\alpha)+1}(-x_1,\alpha) + \tswap(C)^{l(\alpha)+1}(-x_1,\alpha) \right) \\
&=	\tswap(M)^4(-x_1,-x_4,-x_3,-x_2) +\coll^{4}_{3,4}\circ \tswap(M)^{4}(-x_1,-x_4,-x_2,-x_3) \\
&\quad -\tswap(M)^4(-x_1,-x_2,-x_3,-x_4) -\coll^{4}_{3,4}\circ \tswap(M)^{4}(-x_1,-x_2,-x_4,-x_3) \\
&=	\tswap(M)^4(-x_1,-x_4,-x_3,-x_2) +\coll^{4}_{2,3}\circ \tswap(M)^{4}(-x_1,-x_4,-x_3,-x_2) \\
&\quad -\tswap(M)^4(-x_1,-x_2,-x_3,-x_4) -\coll^{4}_{2,3}\circ \tswap(M)^{4}(-x_1,-x_2,-x_3,-x_4).
\end{align*}
In the last equality, we use \eqref{mantar invariant of depth=2}.

By using this and \cite[Lemma 2.5.3]{S-ARIGARI} (see also \cite[Lemma 3.7]{FK})
in a similar way to \eqref{il G70},
we obtain
	\begin{align}\label{eqn:neg circ pus on il}
		&\tswap(M)^4(x_4,x_1,x_2,x_3)
		=\swap(M)^3(x_1-x_4,x_2-x_4,x_3-x_4) \\
		&=M^3(x_3-x_4,x_2-x_3,x_1-x_2)=M^3(x_1-x_2,x_2-x_3,x_3-x_4) \nonumber \\
		&=\swap(M)^3(x_1-x_4,x_1-x_3,x_1-x_2)=\tswap(M)^4(-x_1,-x_4,-x_3,-x_2) \nonumber \\
		&= \tswap(M)^4(-x_1,-x_2,-x_3,-x_4) +\coll^{4}_{2,3}\circ \tswap(M)^{4}(-x_1,-x_2,-x_3,-x_4) \nonumber\\
		&\quad -\coll^{4}_{2,3}\circ \tswap(M)^{4}(-x_1,-x_4,-x_3,-x_2). \nonumber
	\end{align}
Here, by the alternility of $\swap(M)+C$, the right hand side of \eqref{senary relation 2} is calculated to be 
	\begin{align}\label{eqn:pre-senary relation}
		\tswap&(M)^4(x_2,x_3,x_4,x_1) +\coll^{4}_{{1,2}}\circ \tswap(M)^4(x_1,x_2,x_3,x_4) \\
		=&\tswap(M)^4(x_2,x_3,x_4,x_1) +\coll^{4}_{{1,2}}\circ \tswap(M)^4(x_1,x_2,x_3,x_4) \nonumber\\
		&+\sum_{\alpha\in X_{\Z}^\bullet}
		\hSh{(x_2)}{(x_3,x_4)}{\alpha}
		\left( \tswap(M)^{l(\alpha)+1}(x_1,\alpha) 
			+ \tswap(C)^{l(\alpha)+1}(x_1,\alpha) \right) \nonumber\\
		=& \tswap(M)^4(x_1,x_2,x_3,x_4) 
			+\coll^{4}_{{2,3}}\circ \tswap(M)^4(x_1,x_2,x_3,x_4) + 3C_3 \nonumber\\
		&+\{\tswap(M)^{4}(x_1,x_3,x_2,x_4) 
			+\tswap(M)^{4}(x_1,x_3,x_4,x_2)+\tswap(M)^{4}(x_2,x_3,x_4,x_1) \nonumber\\
		&\qquad+\coll^{4}_{{3,4}}\circ \tswap(M)^{4}(x_1,x_3,x_2,x_4)
			+\coll^{4}_{{1,2}}\circ \tswap(M)^{4}(x_1,x_2,x_3,x_4)\}. \nonumber
	\end{align}
	We calculate the last five terms in the above equation.
	By applying \eqref{eqn:neg circ pus on il} to the first three terms, we have
{\small	
	\begin{align*}
		&\tswap(M)^{4}(x_1,x_3,x_2,x_4) +\tswap(M)^{4}(x_1,x_3,x_4,x_2)+\tswap(M)^{4}(x_2,x_3,x_4,x_1) \\
		&\qquad+\coll^{4}_{{3,4}}\circ \tswap(M)^{4}(x_1,x_3,x_2,x_4)	+\coll^{4}_{{1,2}}\circ \tswap(M)^{4}(x_1,x_2,x_3,x_4) \\
		&=\{\tswap(M)^{4}(-x_3,-x_2,-x_4,-x_1) \\
		&\ + \coll^{4}_{{2,3}}\circ \tswap(M)^{4}(-x_3,-x_2,-x_4,-x_1) - \coll^{4}_{{2,3}}\circ \tswap(M)^{4}(-x_3,-x_4,-x_1,-x_2)\} \\
		&\ +\{\tswap(M)^{4}(-x_3,-x_4,-x_2,-x_1) \\
		&\ + \coll^{4}_{{2,3}}\circ \tswap(M)^{4}(-x_3,-x_4,-x_2,-x_1) - \coll^{4}_{{2,3}}\circ \tswap(M)^{4}(-x_3,-x_2,-x_1,-x_4)\} \\
		&\ +\{\tswap(M)^{4}(-x_3,-x_4,-x_1,-x_2) \\
		&\ + \coll^{4}_{{2,3}}\circ \tswap(M)^{4}(-x_3,-x_4,-x_1,-x_2) - \coll^{4}_{{2,3}}\circ \tswap(M)^{4}(-x_3,-x_2,-x_1,-x_4)\} \\
		&\ +\coll^{4}_{{3,4}}\circ \tswap(M)^{4}(x_1,x_3,x_2,x_4)
		+\coll^{4}_{{1,2}}\circ \tswap(M)^{4}(x_1,x_2,x_3,x_4).
\intertext{By rearrangements, we get}
		&=\{\tswap(M)^{4}(-x_3,-x_2,-x_4,-x_1) 
			+ \tswap(M)^{4}(-x_3,-x_4,-x_2,-x_1) \\
		&\quad +\tswap(M)^{4}(-x_3,-x_4,-x_1,-x_2) \\
		&\quad +\coll^{4}_{{2,3}}\circ \tswap(M)^{4}(-x_3,-x_2,-x_4,-x_1) 
			- \coll^{4}_{{2,3}}\circ \tswap(M)^{4}(-x_3,-x_2,-x_1,-x_4) \} \\
		&\quad +\{ \coll^{4}_{{2,3}}\circ \tswap(M)^{4}(-x_3,-x_4,-x_2,-x_1) 
			+\coll^{4}_{{3,4}}\circ \tswap(M)^{4}(x_1,x_3,x_2,x_4) \} \\
		&\quad +\{
			- \coll^{4}_{{2,3}}\circ \tswap(M)^{4}(-x_3,-x_2,-x_1,-x_4)
		+\coll^{4}_{{1,2}}\circ \tswap(M)^{4}(x_1,x_2,x_3,x_4) \}.
\intertext{By applying \eqref{mantar invariant of depth=2} to the fifth term, \eqref{eqn:coll34 from il G70} to the seventh term and \eqref{eqn:coll12 from il G70} to the ninth term, we have}
		&=\{\tswap(M)^{4}(-x_3,-x_2,-x_4,-x_1) +\tswap(M)^{4}(-x_3,-x_4,-x_2,-x_1) \\
		&\quad  +\tswap(M)^{4}(-x_3,-x_4,-x_1,-x_2)\\
		&\quad + \coll^{4}_{{2,3}}\circ \tswap(M)^{4}(-x_3,-x_2,-x_4,-x_1) 
			+\coll^{4}_{{3,4}}\circ \tswap(M)^{4}(-x_3,-x_4,-x_2,-x_1)\} \\
		&\quad +\{ \coll^{4}_{{2,3}}\circ \tswap(M)^{4}(-x_3,-x_4,-x_2,-x_1) 
			-\coll^4_{{2,3}}\circ \tswap(M)^4(-x_3,-x_2,-x_4,-x_1) \} \\
		&\quad -\{ \coll^{4}_{{2,3}}\circ \tswap(M)^{4}(-x_3,-x_2,-x_1,-x_4) 
			+\coll^4_{{3,4}}\circ \tswap(M)^4(-x_3,-x_4,-x_1,-x_2) \}.
\intertext{Since the sum of the first five terms is equal to $-3C_3$ by the alternility of $\swap(M) +C$,
the application of \eqref{mantar invariant of depth=2} to the last term gives}
		&=-3C_3 \\
		&\quad -\{ \coll^{4}_{{2,3}}\circ \tswap(M)^{4}(-x_3,-x_2,-x_1,-x_4) 
			-\coll^4_{{2,3}}\circ \tswap(M)^4(-x_3,-x_1,-x_2,-x_4) \} \\
		&=-3C_3.
	\end{align*}
	Therefore, by applying this equation to \eqref{eqn:pre-senary relation}, we calculate
	\begin{align*}
		&\tswap(M)^4(x_2,x_3,x_4,x_1) +\coll^{4}_{{1,2}}\circ \tswap(M)^4(x_1,x_2,x_3,x_4) \\
		&= \{\tswap(M)^4(x_1,x_2,x_3,x_4) +\coll^{4}_{{2,3}}\circ \tswap(M)^4(x_1,x_2,x_3,x_4) + 3C_3 \}
		-3C_3 \\
		&= \tswap(M)^4(x_1,x_2,x_3,x_4) +\coll^{4}_{{2,3}}\circ \tswap(M)^4(x_1,x_2,x_3,x_4).
	\end{align*}
}
	Hence, we obtain \eqref{senary relation 2} for $r=3$.
\end{proof}

We expect that our above calculations would help us to
construct  a still conjectural inclusion
$\Fil_{\D}^{2}\ARI_{\underline\al\ast\underline\il}\hookrightarrow
\ARI_{\pspush/\pusnu}.
$

\bigskip
Now the proof of Theorem \ref{thm:senary check} is immediate.
Remark \ref{rem:SaSch} says $\ma_{\tilde f}\in \Fil_{\D}^{2}\ARI_{\underline\al\ast\underline\il}$
for $\tilde f\in \dmr$.
By Proposition \ref{appendix theorem}, we see that
$M=\ma_{\tilde f}$ satisfies
Ecalle's senary relation \eqref{senary relation} for $r=1$, $2$ and $3$.
Whence our claim is shown.
\qed

\end{document}